%% file: 2autf42.tex
\documentclass[twoside,12pt]{amsart}

\usepackage[T1]{fontenc}
\usepackage{amsfonts}
\usepackage{amssymb}
\usepackage{amsthm}
\usepackage{amsmath}
\usepackage{graphicx}
\usepackage{url}
\usepackage[all]{xy}
\usepackage{lscape}
\usepackage{url}

\usepackage[shortlabels]{enumitem}
\setlist{leftmargin=*}

\theoremstyle{plain}

\newcommand{\Irr}{\mathrm{Irr}}

\newcommand{\Sp}{\mbox{\rm Sp}}

\newcommand{\St}{\mbox{\rm St}}

\newcommand{\Aut}{\mbox{\rm Aut}}
\newcommand{\Inn}{\mbox{\rm Inn}}

\newcommand{\Ind}{\mbox{\rm Ind}}

\newtheorem*{theorem*}{Theorem}
\newtheorem{num}{Notation}[section]

\newtheorem*{define*}{Definition}

\newtheorem{thm}[num]{Theorem}
\newtheorem*{thm*}{Theorem}
\newtheorem{lem}[num]{Lemma}
\newtheorem*{lem*}{Lemma}

\newtheorem*{prp*}{Proposition}

\newtheorem*{cor*}{Corollary}

\newtheorem*{conj*}{Conjecture}

\newtheorem*{xmpl*}{Example}
\newtheorem{rem}[num]{Remark}
\newtheorem*{rem*}{Remark}

\begin{document}
\date{\today}
\title{The irreducible Brauer characters of the automorphism group of the 
Chevalley group $F_4(2)$}

\author{Gerhard Hiss and Frank L{\"u}beck}

%\address{Lehrstuhl f\"ur Algebra und Zahlentheorie, RWTH Aachen University,
%52056 Aach\-en, Germany}
%
%\email{G.H.: gerhard.hiss@math.rwth-aachen.de}
%\email{F.L.: frank.luebeck@math.rwth-aachen.de}

\thanks{This is a contribution to
Project-ID 286237555 -- TRR 195 -- by the
German Research Foundation}


%\begin{dedication}
%Dedicated to the memory of our friend Richard Parker
%\end{dedication}

\maketitle

\begin{center}
\small Dedicated to the memory of our friend Richard Parker
\end{center}

\begin{abstract}
Using computational methods, we determine the irreducible Brauer characters
of the automorphism group of the Chevalley group $F_4(2)$, up to one parameter
and one consistency issue.
\end{abstract}

\section{Introduction and results}\label{1}

\markright{INTRODUCTION}
  
The purpose of this work, envisaged as a series of two articles, is to compute 
the decomposition matrices of the groups of the form $2.F_4(2).2$, i.e.\ of the 
covering groups of the automorphism group $F_4(2).2$ of the simple Chevalley 
group~$F_4(2)$. In this first part we concentrate on the group $F_4(2).2$. In 
the second part we will consider the covering groups of $F_4(2).2$ and settle
the questions left over here. Since the decomposition matrices of the covering 
group $2.F_4(2)$ of~$F_4(2)$ are known, our objective falls into the following 
general problem.

\subsection{The problem}

Let $\tilde{H}$ be a finite group containing the normal subgroup
$H \unlhd \tilde{H}$ of index~$2$. Choose $x \in \tilde{H} \setminus H$, and 
let~$\sigma$ denote the automorphism of~$H$ induced by conjugation with~$x$.
Further, let~$\varepsilon$ denote the irreducible character of~$\tilde{H}$ with
kernel~$H$.

Let~$p$ be an odd prime and let~$\tilde{B}$ be a $p$-block of~$\tilde{H}$. The 
task is to compute the decomposition matrix of~$\tilde{B}$, provided the 
decomposition matrices of~$H$ are known. For this purpose, let~$B$ be a $p$-block 
of~$H$ covered by~$\tilde{B}$. Then there are three situations, for which we 
summarize some well known results.

\subsubsection{The block~$B$ is not $\sigma$-invariant} 
\label{Case111}
Then~$B$ and~$\tilde{B}$ are Morita equivalent; see 
\cite[Theorem~$6.8.3$]{LinckBook}. In particular, the decomposition matrices 
of~$B$ and~$\tilde{B}$ can be identified via restrictions of characters.

\subsubsection{The block~$B$ is $\sigma$-invariant, and there is a block 
$\tilde{B}' \neq \tilde{B}$ covering~$B$} 
\label{Case112}
Then multiplication by~$\varepsilon$ provides a Morita equivalence 
between~$\tilde{B}$ and~$\tilde{B}'$. Moreover, restriction of characters 
induces a Morita equivalence between~$\tilde{B}$ and~$B$. The first assertion is 
clear, for the second use~\cite[Th\'eor\`eme 2.4]{brouMo}.

\subsubsection{The block~$B$ is $\sigma$-invariant, and~$\tilde{B}$ is the
unique block of~$\tilde{H}$ covering~$B$} 
\label{Case113}
By assumption, we know the PIMs of~$B$. Let~$\Phi$ be one of these. If 
$\Phi \neq \Phi^{\sigma}$, then $\Ind_H^{\tilde{H}}( \Phi )$ is 
indecomposable. Suppose then that $\Phi = \Phi^{\sigma}$. Then 
$$\Ind_H^{\tilde{H}}( \Phi ) = \tilde{\Phi}^+ + \tilde{\Phi}^-,$$
with $\tilde{\Phi}^- = \varepsilon \otimes \tilde{\Phi}^+ \neq \tilde{\Phi}^+$.
Let $\chi, \psi$ be irreducible constituents of~$\Phi$ occurring with
multiplicity~$a$ and~$b$, respectively. Assume that $\chi^{\sigma} \neq \chi$
and that $\psi^{\sigma} = \psi$. Then $\Ind_H^{\tilde{H}}( \chi ) =: 
\tilde{\chi}$ is irreducible and 
$\Ind_H^{\tilde{H}}( \psi ) = \tilde{\psi}^+ + \tilde{\psi}^-$ with distinct
irreducible constituents $\tilde{\psi}^- = \varepsilon \otimes \tilde{\psi}^+ 
\neq \tilde{\psi}^+$. These constituents occur in $\tilde{\Phi}^+$ and 
$\tilde{\Phi}^-$ according to the following scheme.

$$
\begin{array}{l|cc}\hline\hline
& \tilde{\Phi}^+ & \tilde{\Phi}^- \rule[- 0pt]{0pt}{ 13pt} \\ \hline\hline
\tilde{\chi} & a & a \\
\tilde{\psi}^+ & c & \tilde{c} \\
\tilde{\psi}^- & \tilde{c} & c \\ \hline
\end{array}
$$
where $c + \tilde{c} = b$.
The assertions outlined above follow from Clifford theory; for details 
see \cite[Subsection~$2.2$]{BGJ}.

\subsection{The results}
\label{Results}
In the present first part of our paper, we investigate the above problem for
the group $\tilde{H} = \Aut( F_4(2) )$. Let us write $G = F_4(2)$ for the simple
Chevalley group of type~$F_4$ over the field with two elements. Then~$G$ is a
normal subgroup of $\Aut( G )$ of index~$2$, if we identify $\Inn( G )$ 
with~$G$. Adopting Atlas notation, we write $\Aut( F_4(2) ) = G.2$.
The decomposition matrices for~$G$ are known; see~\cite{HF42,BHLL}.

With this notation, we prove the following results, where the tables we refer to 
are collected in the appendix. 

\begin{thm}
\label{Block1Mod3}
The decomposition matrix of the principal $3$-block~$B_1$ of $G.2$ can be 
determined from {\rm Tables~\ref{Mod3B1ApproxI}} and~{\rm \ref{Mod3B1RelsI}}, 
where the parameters in Table~{\rm \ref{Mod3B1ApproxI}} are $a = 3$ and 
$\tilde{a} = 1$.
\end{thm}

Some comments on the statement of Theorem~\ref{Block1Mod3} are appropriate. 
Firstly, the body~$X$ of Table~\ref{Mod3B1ApproxI} gives the decomposition 
numbers for the characters of a special basic set~$\text{\rm BS}_1$ of~$B_1$. 
Secondly, the 
body~$Y$ of Table~\ref{Mod3B1RelsI} contains the expansions of the other 
irreducible characters of~$B_1$ in terms of the basic set characters, on 
restriction to the $3$-regular classes. The decomposition matrix of the 
irreducible characters of~$B_1$ not in the basic set equals $Y^tX$.

\begin{thm}
\label{Block2:9Mod3}
The decomposition matrices of the $3$-blocks~$B_2$ and~$B_9$ of~$G.2$ are as 
given in {\rm Tables~\ref{Mod3B2Decmat}} and{\rm~\ref{Mod3B9Decmat}}, with the 
unknown parameters $b, \tilde{b} \in \{ 0, 1 \}$ and $b + \tilde{b} = 1$.
\end{thm}

\begin{thm}
\label{Block1:2Mod5}
There are two distinct blocks of~$G.2$ covering the principal $5$-block of~$G$.
The decomposition matrices of these two blocks thus agree with the decomposition
matrix of the principal $5$-block of~$G$; see {\rm \ref{Case112}}.
\end{thm}

\begin{thm}
\label{Block1Mod7}
The decomposition matrix of the principal $7$-block of~$G.2$ is given
in {\rm Table~\ref{Mod7B1ApproxI}}, where the parameters take the values
$a = \tilde{d} \in \{ 0, 1 \}$ and $b = c = 1$. For $x \in \{ a, b, c, d \}$, 
we have $\tilde{x} = 1 - x$.
\end{thm}

The proofs of Theorems~\ref{Block1Mod3}--\ref{Block1Mod7} are given in
Sections~\ref{Prime3}--\ref{Condensation} below. Our investigations are based on 
ideas and concepts devised by Richard Parker. The results of 
Sections~\ref{Prime3},~\ref{Prime7} have been obtained with the GAP-package 
MOC~\cite{gapmoc}. This is built upon the original MOC system developed by 
Richard Parker, Klaus Lux and the first author in the 1980s,~see \cite{moc}. 
Section~\ref{Condensation} describes our application of the condensation 
techniques, which go back to Richard Parker and Jon Thackray \cite{Thackray},
and the MeatAxe64 due to Richard Parker~\cite{MA64.bl,MA64}. 

\begin{rem}
{\rm 
\begin{enumerate}[(a), widest=(a)]

\item As $[G.2\colon\!G] = 2$, the irreducible $2$-modular characters of~$G.2$
agree with those of~$G$.

\item The parameter~$b$ in Theorem~\ref{Block2:9Mod3} remains undetermined for 
the time being. The irreducible Brauer characters $5262894^{\pm}$ in the two
cases differ by their values at class~$4$P, which equal~$0$, if $b = 1$, and 
$+/\!\!-\!28$, if $b = 0$. 
By condensing a suitable tensor product or an induced representation, and 
computing the trace at a $4$P-element, one can determine~$b$. We postpone this 
computation to the second paper of our series.

\item The two cases $a = 1$, respectively $a = 0$ in Theorem~\ref{Block1Mod7}
result in two pairs of PIMs for $1377^\pm$, which only differ at the classes 
$32$A, $32$B (in Atlas notation). By swapping these classes in the character 
table of $G.2$ and leaving the order of the irreducible characters fixed, the 
two cases are transformed into one another. In fact there is an automorphism 
$(32\text{\rm A}, 32\text{\rm B})$ of the character table of~$G.2$, which swaps 
the two ordinary characters $\chi_{44}^\pm = 947700^\pm$, and fixes all other 
irreducible characters. With regard to the $p$-modular characters for $p = 7$,
we may thus assume that $a = 1$. This specification may lead to consistency 
issues, whenever the classes $32$A, $32$B are involved in a similar way in 
Brauer character tables of $G.2$ for other primes. This is indeed the case for 
$p = 17$; see \cite[Theorem~$2.5$]{White2}. In order to solve this problem, one 
has to fix a representative~$y$ for class $32$A, say, in terms of a specific 
word in a standard generating system for~$G.2$. Then, the trace on~$y$ of the 
irreducible Brauer characters $5951582^+$ will decide the question. Again, 
this computation is deferred to our next paper.

\item The blocks considered in Theorems~\ref{Block1Mod3}--\ref{Block1Mod7} 
account for all blocks of~$G.2$ with non-cyclic defect groups. The decomposition 
matrices of $G.2$ for blocks with cyclic defect groups have been computed by 
Donald White in \cite{White2}.
\end{enumerate}
}
\end{rem}

As a preparation for this work, Thomas Breuer computed the character tables
of all maximal subgroups of $G.2$, using Magma~\cite{magma} and the permutation
representation of $G.2$ of degree $139776$, provided by the Atlas of Finite 
Group Representations~\cite{WWWW}, now part of the AtlasRep package of 
GAP~\cite{WPNBB24}.

%\subsection{Dedication}
%\marginpar{Revise the paragraph on dedication.}
%This work is dedicated to our friend Richard Parker. Almost all of the ideas
%and techniques on which our investigations are based, are due to him. 
%His lifelong commitment to bringing his genuine invention, the MeatAxe, to 
%perfection, has culminated in the MeatAxe64, of which this work makes
%substantial use. The origins of the MOC system date back to the 1980s, when
%Richard was a regular visitor at our Aachen department. Richard developed the 
%strategy and the theory \cite{moc} behind the system and wrote large parts of 
%the code, assisted by Klaus Lux and the first author of this paper. Later, 
%Lukas Maas and Felix Noeske transformed MOC into a GAP package \cite{gapmoc}. 
%It is reassuring that these old ideas of Richard are still vital in 
%contemporary research.

\section{Notation and preliminaries}

\markright{NOTATION AND PRELIMINARIES}

Recall the notation introduced at the beginning of Subsection~\ref{Results}.
Thus~$G$ denotes the simple Chevalley group~$F_4(2)$. The ordinary character 
table of~$G$ is contained in the Atlas \cite[Pages~$167$--$169$]{Atlas}. This 
table is also available in GAP~\cite{GAP4}, and serves as a basis for our 
computations.

We use the numbering of the irreducible characters of~$G$ given in the Atlas;
this agrees with the numbering in GAP. Thus $\Irr( G ) = \{ \chi_1, \ldots ,
\chi_{95} \}$. As usual, an irreducible character of~$G$ is sometimes denoted
by its degree.

The group~$G$ is contained, as an index~$2$ subgroup, in the automorphism group 
$\Aut(F_4(2))$ of $F_4(2)$, and we write $G.2 = \Aut(F_4(2))$. Let 
$x \in G.2$ denote the exceptional automorphism of~$F_4(2)$ of order~$2$.

\subsection{Labelling the blocks}
\label{TheBlocks}

We label the blocks of~$G$ and~$G.2$ as in GAP; this labelling is also used in 
the database of decomposition matrices of the Modular Atlas home 
page~\cite{ModAtlHomePageDec}. We write~$B_i$ for the $p$-block of~$G$ which is 
the block with number~$i$ in GAP. 

%\marginpar{The explanation of the GAP numbering of the blocks can be omitted.}
%Let us shortly explain the GAP numbering of blocks. Suppose~$H$ is a finite 
%group and the GAP-character table of~$H$ contains the rows $\psi_1, \ldots, 
%\psi_k$, in this order. For a prime $p$ and and integer $i \geq 1$, the 
%$p$-block $B_i$ is determined by the following conditions: $\psi_i \in B_i$ 
%and $\psi_j \not\in B_i$ for all $1 \leq j < i$. In particular, if~$\psi_1$ is
%the trivial character,~$B_1$ is the principal block.

\subsection{Labelling the ordinary irreducible characters of $G.2$}
\label{TheCharacters}
\label{LabellingOrdinaries}

Let~$\sigma$ denote the automorphism of~$G$ induced by conjugation with~$x$, and 
let~$\varepsilon$ be the unique irreducible character of~$G.2$ with kernel~$G$.

Let $\chi := \chi_i \in \Irr( G )$. If $\chi^{\sigma} \neq \chi$, then 
$\Ind_G^{G.2}( \chi ) \in \Irr( G.2 )$, and we
write 
$$\chi_i^{0} := \Ind_G^{G.2}( \chi ),$$
so that $\chi_i^{0} = \chi_j^{0}$ if $\chi_i^{\sigma} = \chi_j$.

Otherwise $\Ind_G^{G.2}( \chi )$ is the sum of two distinct irreducible
extensions of~$\chi$ to~$G.2$, and we write 
$$\Ind_G^{\hat{G}}( \chi ) = \chi_i^+ + \chi_i^-.$$
Then $\chi_i^- = \varepsilon \otimes \chi_i^+$. In particular, 
$\chi_i^{-}( x ) = - \chi_i^{+}( x )$. If these numbers are real and non-zero, 
we choose the notation such that $\chi_i^+( x ) > 0$. This condition is not 
satisfied exactly for the two characters~$\chi_{13}$ and~$\chi_{44}$. In these 
cases, we choose $\chi_i^+$ as one of the two extension of~$\chi_i$, at random. 
The symbol $\chi_i^{\pm}$ denotes the pair $(\chi_i^+,\chi_i^-)$.

The above conventions are expanded analogously to the case when 
$\chi \in \Irr(G)$ is denoted by its degree.

\subsection{Character table automorphisms} In the course of our proofs we induce 
projective characters from the maximal subgroups $[2^{20}]\colon\!A_6.2^2$ and 
$[2^{22}]\colon\!\!(S_3 \times S_3)\colon\!2\!$ of~$G.2$. In each case, there
is one orbit of class fusions under the character table automorphisms of the 
subgroup and those of~$G.2$, and we choose one fusion out of its orbit. We have 
to make sure that our arguments are independent 
of the chosen fusion map. For this we check that the projective characters of 
the subgroups we induce are invariant under all table automorphisms of the 
subgroup. In turn we check that the induced characters are invariant under 
table automorphisms of~$G.2$. Hence any other possible fusion will yield the
same projective characters.

\section{The prime $3$}
\label{Prime3}

\markright{THE PRIME $3$}

Here, we prove the elementary parts of Theorems~\ref{Block1Mod3} 
and~\ref{Block2:9Mod3}. 

\subsection{The principal block} Let $B_1$ denote the principal $3$-block 
of~$G.2$. Then $k(B_1) = 56$ and $\ell(B_1) = 31$. The restrictions to the 
$3$-regular classes of the following characters yields a special basic set 
$\text{\rm BS}_1$ of~$B_1$:
%
% Positions of basic set characters in Irr( CharacterTable( "F4(2).2" ) );
%[1,2,3,4,5,6,7,10,11,12,13,14,16,17,18,24,25,28,29,30,33,34,41,53,54,55,57,61,62,88,89];
%
$\chi_{1}^{\pm}$, $\chi_{2}^{\pm}$, $\chi^0_{3}$, $\chi_{5}^{\pm}$,
$\chi_{7}^{\pm}$, $\chi_{8}^{\pm}$, $\chi^0_{9}$, $\chi_{13}^{\pm}$,
$\chi^0_{14}$, $\chi_{20}^{\pm}$, $\chi^0_{22}$, $\chi_{24}^{\pm}$,
$\chi_{26}^{\pm}$, $\chi^0_{33}$, $\chi_{45}^{\pm}$, $\chi^0_{46}$, 
$\chi^0_{50}$, $\chi_{58}^{\pm}$, $\chi_{94}^{\pm}$.
The relations displayed in Table~\ref{Mod3B1RelsI} prove this assertion.

Table~\ref{Mod3B1Projs} gives some projective characters of~$B_1$, restricted 
to~$\text{\rm BS}_1$. These either arise as induced PIMs of the subgroup
$[2^{22}]\colon\!\!(S_3 \times S_3)\colon\!\!2$ of~$G.2$ or as products of 
defect~$0$ characters of~$G.2$ with ordinary characters. The list below gives 
more details; in the latter case, the defect~$0$ character is the first factor 
in the displayed product, in the former case the PIMs of 
$[2^{22}]\colon\!\!(S_3 \times S_3)\colon\!2\!$ are given as sums of the 
ordinary irreducible characters $\theta_i$, $i = 1, \ldots , 384$, of 
$[2^{22}]\colon\!\!(S_3 \times S_3)\colon\!2\!$ in the ordering of the 
GAP-character table. The superscript in square brackets refers to the fact that
we are inducing characters from the fourth maximal subgroup.

$$
\begin{array}{r|c} \hline\hline
\Phi & \text{Origin} \rule[- 3pt]{0pt}{ 14pt} \\ \hline
{32} & \theta^{[4]}_{1} + \theta^{[4]}_{6} + \theta^{[4]}_{7} \rule[ 0pt]{0pt}{ 13pt} \\
{33} & \chi_{27}^+ \otimes 833^+ \\
{34} & \theta^{[4]}_{11} + \theta^{[4]}_{31} \\
{35} & \theta^{[4]}_{41}  \\
{36} & \theta^{[4]}_{4} + \theta^{[4]}_{7} + \theta^{[4]}_{9} \\
{37} & \theta^{[4]}_{92} + \theta^{[4]}_{98} + \theta^{[4]}_{100} \\
{38} & \theta^{[4]}_{15} + \theta^{[4]}_{32} \\
{39} & \chi_{27}^+ \otimes 1326^+ \\
{40} & \chi_{16}^+ \otimes 1326^+ 
\rule[- 2pt]{0pt}{10pt} \\ \hline\hline
\end{array}
$$

We now derive the parametrized decomposition matrix given in 
Table~\ref{Mod3B1ApproxI}, up to the parameters~$a$ and~$\tilde{a}$. These are 
non-negative integers such that $a + \tilde{a} = 4$. The PIMs of~$B_1$ are 
denoted by $\Phi_1, \ldots, \Phi_{31}$. First, we induce the PIMs of the 
principal block of~$G$ to~$G.2$. The non-invariant PIMs of this block yield the 
PIMs $\Phi_5, \Phi_{12}, \Phi_{15}$, $\Phi_{18}$, $\Phi_{23}$, $\Phi_{26}$ and 
$\Phi_{27}$ of~$B_1$. The invariant PIMs of the principal block of~$G$ induce to 
pairs of splitting PIMs, i.e.\ the two elements of such a pair differ by 
multiplication with~$\varepsilon$. The splitting follows the rules noted 
in~\ref{Case113}.

Observe that 
$\Phi_{40}$ is a PIM, which is one member of a pair of splitting PIMs, thus
yielding $\Phi_{28}$ and $\Phi_{29}$. Similarly, $\Phi_{39}$ and $\Phi_{38}$
determine the splitting of $\Phi_{24}$ and $\Phi_{25}$, respectively
$\Phi_{19}$ and $\Phi_{20}$. Also, $\Phi_{37}$ yields the given form of
$\Phi_{16}$ and $\Phi_{17}$. The splitting of $\Phi_{13}$ and $\Phi_{14}$ 
follows from the fact that the two liftable Brauer characters of degree $63700$
are complex conjugates, whereas all other constituents in the corresponding PIMs 
are real valued.  Since $\Phi_{36}$ contains the PIM $\Phi_{21}$,
the former projective character determines the splitting of $\Phi_{10}$ and
$\Phi_{11}$, as well as that of $\Phi_{21}$ and $\Phi_{22}$. The form of $\Phi_8$ 
and $\Phi_9$ is proved with the help of $\Phi_{35}$ and $\Phi_{29}$.
Similarly, $\Phi_{34}$ and $\Phi_{33}$
determine the splitting of $\Phi_{6}$ and $\Phi_{7}$, respectively
$\Phi_{3}$ and $\Phi_{4}$. Finally, $\Phi_1 = \Phi_{32} - 2\Phi_5 -
3 \Phi_{10} - 2 \Phi_{11}$, where the multiplicities of the PIMs to be subtracted 
from $\Phi_{32}$ are implied by the triangular shape of the decomposition matrix,
already proved. This completes the proof for the parametrized decomposition
matrix for the principal $3$-block of~$G.2$. It remains to determine the
parameter~$a$. This will be achieved in Section~\ref{Condensation}.

\subsection{The non-principal blocks of non-cyclic defect}
These are the two blocks~$B_2$ and~$B_9$ in the numbering of GAP. Let~$B$ be 
one of
the blocks~$B_2$ or~$B_9$. Then $k(B) = 9$ and $\ell(B) = 7$, and we number the 
PIMs of~$B$ by $\Phi_1, \ldots, \Phi_7$. There is a unique block of~$G$ covered 
by~$B$, whose decomposition matrix is given in~\cite{HF42}. The PIM $\Phi_5$
is obtained by inducing a non-invariant PIM of~$G$.

Further PIMs of~$B$ are obtained according to the following lists.

$$
\begin{array}{r|c} \hline\hline
\multicolumn{2}{c}{B_2} \rule[- 3pt]{0pt}{ 14pt} \\ \hline\hline
\Phi & \text{Origin} \rule[- 3pt]{0pt}{ 14pt} \\ \hline
{1} & \theta^{[2]}_{6} \rule[ 0pt]{0pt}{ 15pt} \\
{2} & \theta^{[2]}_{8} \\
{3} & \theta^{[4]}_{66} \\
{4} & \theta^{[4]}_{62}  \\
{6} & \chi_{16}^+ \otimes 833^+ \\
{7} & \chi_{16}^- \otimes 833^+ 
\rule[- 2pt]{0pt}{10pt} \\ \hline\hline
\end{array}
\quad\quad
\begin{array}{r|c} \hline\hline
\multicolumn{2}{c}{B_9} \rule[- 3pt]{0pt}{ 14pt} \\ \hline\hline
\Phi & \text{Origin} \rule[- 3pt]{0pt}{ 14pt} \\ \hline
{3} & \theta^{[4]}_{18} \rule[ 0pt]{0pt}{ 15pt} \\
{4} & \theta^{[4]}_{28}  \\
{6} & \theta^{[4]}_{26} \\
{7} & \theta^{[4]}_{22} 
\rule[- 2pt]{0pt}{10pt} \\ \hline\hline
\end{array}
$$
Here, the irreducible characters of the maximal subgroup 
$[2^{20}]\colon\!A_6.2^2$ are denoted by $\theta_i^{[2]}$, 
$i = 1, \ldots , 331$, and those of the maximal subgroup 
$[2^{22}]\colon\!\!(S_3 \times S_3)\colon\!\!2$ by 
$\theta_i^{[4]}$, $i = 1, \ldots, 384$. The characters induced are all 
of $3$-defect~$0$. The splitting of the PIMs $\Phi_1$, $\Phi_2$ of 
block~$B_9$ follows the rules formulated in~\ref{Case113}.

\section{The prime $7$}
\label{Prime7}

\markright{THE PRIME $7$}
The columns labelled by $\Phi_1, \ldots , \Phi_{24}$ of 
Table~\ref{Mod7B1ApproxI} constitute the para\-met\-riz\-ed 
decomposition matrix of the principal $7$-block of~$G.2$. The parameters 
$a, b, c, d$ take the values~$0$ or~$1$, and $\tilde{x} = 1 - x$ for 
$x \in \{ a, b, c, d \}$. The PIMs 
$\Phi_3, \Phi_8, \Phi_9$, $\Phi_{16}$, $\Phi_{19}$ and $\Phi_{20}$ are induced 
from PIMs of~$G$. There are PIMs of the form 
$$\Phi_i := (833^+ \otimes \chi_j^{\pm})_{B_1},$$
with~$i, j$ and~$\pm$ as in the following list.

$$
\begin{array}{r|cccccccccc}
i & 1 & 2 & 4 & 5 & 14 & 15 & 21 & 22 & 23 & 24  \\ \hline
j & 2 & 2 & 7 & 7 & 24 & 24 & 16 & 16 & 26 & 26 \\ \hline
\pm & + & - & + & - & + & - & + & - & + & -
\end{array}
$$

The parametrized PIMs are pairs of splitting PIMs, i.e.\ the two elements of 
such a pair differ by multiplication with~$\varepsilon$. There are two pairs of 
complex conjugate PIMs, $(\Phi_{10}, \Phi_{12})$ and $(\Phi_{11}, \Phi_{13})$.
All other PIMs are real.

Finally, there are two projective characters
$$\Phi_{25} := (\chi_8^+ \otimes \chi_6^+)_{B_1}$$
and
$$\Phi_{26} := (\chi_7^+ \otimes \chi_{20}^+)_{B_1}.$$
Notice that $$\Phi_{25} = \Phi_6 + \tilde{a}\Phi_{17} + a \Phi_{18}.$$
Since $5640192^-$ is not a constituent of $\Phi_{25}$, we get
$\tilde{a}\tilde{d} + ad = 0$, which implies that $d = \tilde{a}$.
Also,
$$\Phi_{26} = \Phi_{10} + \Phi_{12} + \Phi_{19} + 4 \Phi_{20}
+(3 - 2b)\Phi_{21} + (1 - 2\tilde{b})\Phi_{22} + x \Phi_{23} + y \Phi_{24}$$
for some non-negative integers $x, y$. This implies $1 - 2\tilde{b} \geq 0$
and thus $b = 1$.

It remains to determine the parameter~$c$. This is achieved in 
Section~\ref{Condensation} below.

\section{Condensation} 
\label{Condensation}

\markright{CONDENSATION}

To determine the parameters~$a$ in 
Theorem~\ref{Block1Mod3} and~$c$ in Theorem~\ref{Block1Mod7}, we use 
condensation techniques. 
Let $\mathbb{F}$ be one of $\mathbb{F}_3$ or $\mathbb{F}_7$ and let~$p$ denote
the characteristic of~$\mathbb{F}$. For $\chi \in \Irr( G.2 )$ we write 
$\chi^{\circ}$ for the restriction of~$\chi$ to the $p$-regular classes.

\subsection{The setup}
\label{TheSetup}
Assume that $(B,N)$ is a split $BN$-pair for~$G$. Then $U := B$ is a Sylow 
$2$-subgroup of~$G$ and $W := N$ is the Weyl group of~$G$. The corresponding 
Dynkin diagram of~$G$ equals

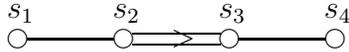
\begin{center}
\begin{picture}(123.5,60)(0,-22)
\put(   0,   0){\circle{7}}
\put(  40,   0){\circle{7}}
\put(  80,   0){\circle{7}}
\put( 120,   0){\circle{7}}
\put( 3.5,   0){\line(1,0){ 33}}
\put(43.5,   2){\line(1,0){ 33}}
\put(43.5,  -2){\line(1,0){ 33}}
\put(83.5,   0){\line(1,0){ 33}}
\put(58.0, -2.9){$>$}
\put(   -4,+8){$s_1$}
\put(   36,+8){$s_2$}
\put(   76,+8){$s_3$}
\put(  116,+8){$s_4$}
\end{picture}
\end{center}
where $s_1, \ldots , s_4$ denote the Coxeter generators of~$W$. The standard 
graph automorphism~$\sigma$ of~$G$ (see 
\cite[Theorem~$1.15.2$(a), Definition $1.15.5$(e)]{GLS}) fixes~$U$ 
and~$W$, and swaps~$s_1$ with~$s_4$ and~$s_2$ with~$s_3$. Recall that 
$x \in G.2$ is an involution which induces~$\sigma$ by conjugation. Thus~$x$ 
normalizes~$U$ and~$W$ and $s_1^x = s_4$, $s_2^x = s_3$. In particular,~$x$
preserves the length of the elements of~$W$. As in~\cite{BHLL}, let 
$$e := \left( \sum_{u \in U} u \right) 
\left( \sum_{w \in W} (-1)^{\ell(w)} w \right) \in \mathbb{F}G$$
denote the Steinberg element. Then~$x$ centralizes~$e$.

Put $\St := e\mathbb{F}G$. Then $\St$ is an $\mathbb{F}G$-module with character 
$\chi_{94}$, the Steinberg character of~$G$. Also,~$\St$ has $\mathbb{F}$-basis
$\{ eu \mid u \in U \}$. Now $G.2 = \langle G, x \rangle$, and~$x$ 
centralizes~$e$, so~$\St$ is an $\mathbb{F}G.2$-module, denoted by~$\St^+$. 
As~$x$ permutes the basis $\{ eu \mid u \in U \}$ with exactly~$4096$ fixed
points, the character of~$\St^+$ equals~$\chi_{94}^+$.

Now~$a$, respectively~$c$ is the multiplicity of the irreducible Brauer character 
$\chi_{20}^{+,\circ}$ in $\chi_{94}^{+,\circ}$. To determine this multiplicity, 
we condense~$\St^+$, using the same condensation subgroup~$V$ as in~\cite{BHLL}; 
that is,~$V$ is the center of the unipotent radical of the standard parabolic 
subgroup~$P$ of type~$C_3$ of~$G$. In particular, $V \leq G$. The condensation 
idempotent is $\iota := 1/|V| \sum_{v \in V} v$.
In this setup, simple $\mathbb{F}G.2$-modules with characters 
$\chi_{20}^{\pm,\circ}$ and $\chi_{21}^{\pm,\circ}$ condense to characters of 
degree~$720$, and no other constituent of $\St^+$ condenses to a character of
this degree.

Condensing the same elements of~$G$ as in~\cite[Subsection~$2.4$]{BHLL}, we can 
use the formulas given in~\cite[Subsections~$2.2$, $2.3$]{BHLL}, with 
$\mathbb{F}_3$ replaced by~$\mathbb{F}$. As remarked 
in~\cite[Subsection~$2.4$]{BHLL}, these condensed elements are algebra 
generators of~$\iota \mathbb{F}G \iota$.
In addition, we condense
$$y_1 := x s_4 s_3 s_2 s_1 \in G.2 \setminus G$$
and
$$y_2 := x s_1 s_2 s_3 s_4 \in G.2 \setminus G.$$ 
The action of~$x$ on the basis $\{ eu \mid u \in U \}$ of~$\St^+$ is given by
applying~$\sigma$ to these basis elements, that is we need to know how~$\sigma$ 
permutes the elements of~$U$. We first determine how~$\sigma$ permutes the root 
subgroups~$U_\alpha$ for positive roots~$\alpha$ (and so the non-trivial 
elements in these subgroups). We write the reflection~$w_\alpha$ along a 
root~$\alpha$ as word in the Coxeter generators~$s_i$, $i=1,2,3,4$, of~$W$. We 
know how~$x$ permutes the~$s_i$ and so we find the positive root~$\beta$ such 
that $w_\beta = w_\alpha^x$. Then~$\sigma$ maps~$U_\alpha$ to~$U_\beta$.

Using this map, symbolic descriptions of the elements of~$U$ as products of root 
elements in some fixed ordering of the roots and the commutator relations 
in~$U$, we obtain a symbolic description of the permutation~$\sigma$ on~$U$.

Thus the matrix
coefficients for the action of $\iota y_i \iota$ on $\St^+ \iota$ for $i = 1, 2$
can be computed as in the remarks of \cite[Subsection~$2.4$]{BHLL}.

\subsection{Applying the trace formula}
In order to determine the multiplicity of the composition factor~$M^+$ 
of~$\St^+$ with Brauer character $\chi_{20}^{+, \circ}$, we need to find a 
suitable element $y \in G.2$, such that the trace of $\iota y \iota$ on 
$\iota M^+ \iota$ is non-zero. Let~$y_1$,~$y_2$ be the elements defined in 
Subsection~\ref{TheSetup}.

\addtocounter{num}{2}
\begin{lem}
\label{Average}
{\rm (a)} The set $Z_i := \{ y_i v \mid v \in V \}$ distributes among the 
conjugacy classes of~$G.2$ according to the following tables, where we use 
Atlas notation for the conjugacy classes.
$$Z_1\!\!:\quad\quad
\begin{array}{r|ccccc} \hline\hline
\text{\rm Class} & 2\text{\rm E} & 4\text{\rm P} & 4\text{\rm Q} & 
                   8\text{\rm L}/\text{\rm M} & 8\text{\rm R} \rule[- 3pt]{0pt}{ 16pt} \\ \hline
\text{\rm No.} & 1 & 3 & 28 & 48 & 48 \\ \hline\hline
\end{array}
$$
$$Z_2\!\!:\quad\quad
\begin{array}{r|cccc} \hline\hline
\text{\rm Class} & 2\text{\rm E} & 4\text{\rm P} & 4\text{\rm Q} & 
                   8\text{\rm L}/\text{\rm M} \rule[- 3pt]{0pt}{ 16pt} \\ \hline
\text{\rm No.} & 2 & 22 & 40 & 64 \\ \hline\hline
\end{array}
$$
{\rm (b)} Let $\chi^+$ be one of $\chi_{20}^{+}$ or $\chi_{21}^{+}$.
Then $\sum_{v \in V} \chi^+( y_1v ) = 128$ and
$\sum_{v \in V} \chi^+( y_2v ) = 0$.

{\rm (c)} Let~$M$ be a composition factor of~$\St^+$ with 
$\varphi_M \in \{ \chi_{20}^{+,\circ}, \chi_{21}^{+,\circ} \}$. Denote the 
$\iota \mathbb{F}G.2 \iota$-character of~$M \iota$ by 
$\varphi^{\iota}_{M\iota}$. Then 
$\varphi^{\iota}_{M\iota}( \iota y_1 \iota ) = 1$ and
$\varphi^{\iota}_{M\iota}( \iota y_2 \iota ) = 0$.
\end{lem}
\begin{proof}
{\rm (a)} This is computed with GAP, using the permutation representation 
of~$G.2$ on $139776$ points, provided by the AtlasRep-package; 
see~\cite{WPNBB24}.
The first generator of~$G.2$ in this representation is an involution in 
$G.2 \setminus G$. As all these involutions are conjugate by an element of~$G$,
we may set~$x$ equal to this first generator. The task now is to identify the
elements $s_1, \ldots , s_4$; in other words, we have to solve a constructive
recognition problem.

In a first step, we construct an $x$-invariant Sylow $2$-subgroup~$U'$ of~$G$.
As all $x$-invariant Sylow $2$-subgroups of~$G$ are conjugate by an element
of $C_G( x )$, we may assume that $U' = U = B$ is the first component of the
$BN$-pair of~$G$ we are looking for. In our situation, $N = W$ is the 
corresponding $x$-invariant Weyl group, generated by the Coxeter generators
$s_1, \ldots, s_4$, with $s_1 = s_4^x$ and $s_2 = s_3^x$. The standard parabolic 
subgroup~$P$ of type~$C_3$ corresponding to $s_2, s_3, s_4$ is the centralizer
of a central element of~$U$. We may thus construct~$P$. Now $P = L \ltimes U_P$
with $L \cong \Sp_6( 2 )$. It is easily checked that
$|C_L( s_4 )| = 4608$. The character tables of~$P$ and~$G$ and the corresponding
class fusions are available in GAP. Using these, we conclude that~$s_4$ lies in
the unique conjugacy class~$C$ of~$P$ of length~$5040$, which fuses into a 
conjugacy class of~$G$ with centralizer order~$47563407360$. The class~$C$ is 
easily located in~$P$. Also, $s_3 \in C$, since~$s_4$ and $s_3$ are conjugate in 
$\langle s_3, s_4 \rangle \leq P$. Thus $(s_4,s_3) \in \mathcal{C}$ with
$\mathcal{C}$ the set of all elements $(s,t) \in C \times C$ satisfying the
following conditions:

\begin{itemize}
\item $|st| = 3$;
\item $|ss^x| = 2$;
\item $|st^x| = 2$;
\item $|tt^x| = 4$;
\item $|U \cap U^{s}| = 2^{23}$;
\item $|U \cap U^{t}| = 2^{23}$;
\item $C_{C_U(x)}( \langle s, t \rangle ) = \{ 1 \}$.
\end{itemize}
The first four conditions arise from the presentation of~$W$, the next two
conditions from the $BN$-axioms. The last condition follows from 
$C_U( W ) = \{ 1 \}$. Using GAP, we find that $|\mathcal{C}| = 4096$.
As $C_U( x )$, a group of order~$4096$, acts on $C \times C$ and preserves 
the above conditions,~$\mathcal{C}$ is a regular $C_U(x)$-orbit on 
$C \times C$.

For $(s,t) \in \mathcal{C}$, put
\begin{equation}
\label{PotentialAverageSet}
Z_{s,t} := \{ x s t t^x s^x v \mid v \in V \}
\end{equation}
and 
\begin{equation}
\label{PotentialAverageSetInv}
Z'_{s,t} := \{ x s^x t^x t s v \mid v \in V \}
\end{equation}
As~$V$ is a normal subgroup of~$U$, the sets $Z_{s,t}$ are conjugate by
elements of $C_U(x)$ as $(s,t)$ varies over~$\mathcal{C}$, and an analogous
remark applies to the sets $Z'_{s,t}$.
By the order distribution of a subset of~$G.2$, we understand the multiset of
orders of the elements of this set. By what we have already said, the order 
distributions of the sets~$Z_{s,t}$ are the same for all $(s,t) \in \mathcal{C}$, 
and likewise for the sets~$Z'_{s,t}$.

We find the order distribution 
$$\{ 2^1, 4^{31}, 8^{96} \}$$
for the sets~$Z_{s,t}$ and
$$\{ 2^2, 4^{62}, 8^{64} \}$$
for the sets~$Z'_{s,t}$, where the multiplicity of an element is indicated
by a superscript.

The elements of order~$4$ of $G.2 \setminus G$ are distinguished by the 
orders of their centralizer in~$G$. The 
same is true for the elements of order~$8$, except if the centralizer has
order~$1280$ or~$128$, respectively. In the former case, the elements lie in
one of the classes~$8$L or~$8$M. In the latter case, in one of the classes
$8$Q or $8$R, which are, however, distinguished by the centralizer order of 
their squares. We find the multiset of centralizer orders
$$ \{ 35942400^1, 20480^3, 3072^{28}, 1280^{48}, 128^{48} \}$$
for $Z_{s,t}$, respectively
$$\{ 35942400^2, 20480^{22}, 3072^{40}, 1280^{64} \}$$
for~$Z'_{s,t}$. In the former case, all elements with centralizer order $128$
square to elements with centralizer order~$16384$, so lie in the class~$8$R. A 
look into the Atlas completes the proof.

(b) This follows from (a) and the character values of~$\chi^{+}$ on the 
respective elements.

(c) This follows from (b) and the trace formula 
$$\varphi^{\iota}_{M\iota}( \iota y_i \iota ) = 
\frac{1}{|V|} \sum_{v \in V}\varphi_M( y_iv )$$
given in~\cite[Subsection~$3.6$]{MNRW}. 
\end{proof}

\addtocounter{subsection}{1}
\subsection{Results of the condensation}
We apply the MeatAxe with generators of the algebra 
$\langle \iota \mathbb{F}G \iota, \iota y_1 \iota, \iota y_2 \iota \rangle$. 
Although the latter might be strictly contained in $\iota \mathbb{F}G.2 \iota$,
this suffices for our purposes. Indeed, the $720$-dimensional 
$\iota \mathbb{F}G.2 \iota$-composition factors of $\St^+\iota$ are irreducible 
as $\iota \mathbb{F}G \iota$-modules, and we can 
compute the traces of $\iota y_1 \iota$ and $\iota y_2 \iota$  on them.

Suppose first that $p = 3$. Then the module $\St^+ \iota$ has 
four composition factors of degrees $720$, which come in two isomorphism types
$720a$ and $720b$ and multiplicities~$1$ and~$3$, respectively. Moreover, the
traces of $\iota y_1 \iota$ on~$720a$ and~$720b$ are~$-1$ and~$1$, 
respectively. By Lemma~\ref{Average}(b) and Table~\ref{Mod3B1ApproxI}, we 
obtain $a = 3$.

Now suppose that $p = 7$, Then $\St^+ \iota$ has composition factors of 
degrees 
$$ 720, 720, 3711, 18555, 39900, 67466, $$
where the two composition factors of degree $720$ are not isomorphic. In fact,
these correspond to the composition factors of $\St^+$ with Brauer characters
$\chi_{20}^{+,\circ}$ and $\chi_{21}^{+,\circ}$, if $c = 1$, respectively
$\chi_{20}^{-,\circ}$ and $\chi_{21}^{-,\circ}$, if $c = 0$.
The trace of $\iota y_1 \iota$ on each of these two composition factors equals~$1$.
By Lemma~\ref{Average}(b) and Table~\ref{Mod7B1ApproxI}, we obtain $c = 1$.

\section*{Acknowledgements} 

It is our pleasure to thank Thomas Breuer for computing the character
tables of the maximal subgroups of $\Aut(F_4(2))$, and Klaus Lux for his
invaluable help with the MeatAxe64 and the MOC system. We also thank Frank 
Himstedt for independently checking the results of Lemma~\ref{Average}(a).

\bigskip
\bigskip

{\small
\noindent\textsc{Lehrstuhl f\"ur Algebra und Zahlentheorie, RWTH Aachen University,
52056 Aach\-en, Germany}

\noindent\textit{Email address: }\texttt{G.H.: gerhard.hiss@math.rwth-aachen.de}

\noindent\textit{Email address: }\texttt{F.L.: frank.luebeck@math.rwth-aachen.de}
}

\newpage

%\cleardoublepage

\input{2autf42Tables}

\end{document}

%% file: 2autf42Tables.tex
\clearpage

\appendix

\section{Tables}
The notation in the following tables is introduced in 
Subsections~\ref{TheBlocks} and~\ref{TheCharacters}. The parameters
in Tables~\ref{Mod3B1ApproxI},~\ref{Mod3B9Decmat}
and~\ref{Mod7B1ApproxI} are discussed in Subsection~\ref{Results}.

\begin{table}[h]
\caption{\label{Mod3B1RelsI} The relations in the principal $3$-block
with respect to the basic set $\text{\rm BS}_1$}
$$
{\footnotesize
\begin{array}{r|rrrrrrrrrrrrr} \hline\hline
%\chi & \multicolumn{1}{c}{\begin{array}{c} 11 \\ 12 \end{array}} % 88400
%     & \multicolumn{1}{c}{\begin{array}{c} 18 \\ 19 \end{array}} % 324870
%     & \multicolumn{1}{c}{21^-}                                  % 183600
%     & \multicolumn{1}{c}{22^-}                                  % 183600
%     & \multicolumn{1}{c}{\begin{array}{c} 28 \\ 29 \end{array}} % 696150
%     & \multicolumn{1}{c}{32^+}                                  % 519792
%     & \multicolumn{1}{c}{32^-}                                  % 519792
%     & \begin{array}{c} 35 \\ 36 \end{array} % 1082900
%     & \multicolumn{1}{c}{37^+}                                  % 584766
%     & \multicolumn{1}{c}{37^-}                                  % 584766
%     & \multicolumn{1}{c}{\begin{array}{c} 40 \\ 41 \end{array}} % 1624350
%     & \multicolumn{1}{c}{\begin{array}{c} 52 \\ 53 \end{array}} % 4331600
%     & \multicolumn{1}{c}{\begin{array}{c} 54 \\ 55 \end{array}} % 5569200
\chi & \chi_{11}^0 & \chi_{18}^0 & \chi_{21}^- & \chi_{22}^- & 
                     \chi_{28}^0 & \chi_{32}^+ & \chi_{32}^- &
                     \chi_{35}^0 & \chi_{37}^+ & \chi_{37}^- &
                     \chi_{40}^0 & \chi_{52}^0 & \chi_{54}^0 
\rule[- 4pt]{0pt}{ 14pt} \\ \hline     
 1^+          &  1& -1 & . & . & -2& . & 1 & -1& 1 & . & -1& 1 & -1 \rule[ 0pt]{0pt}{ 10pt} \\
 1^-          &  1& -1 & . & . & -2& 1 & . & -1& . & 1 & -1& 1 & -1\\
 833^+        & -1&  1 & . & . & 1 & -1& -1& . & -1& -1& 1 & -2& . \\
 833^-        & -1&  1 & . & . & 1 & -1& -1& . & -1& -1& 1 & -2& . \\
 2210^0\,     & -1&  . & . & . & 2 & -1& -1& . & -1& -1& . & -2& . \\
 1326^+       & . &  . & . & . & 1 & . & . & 1 & . & -1& . & 1 & . \\
 1326^-       & . &  . & . & . & 1 & . & . & 1 & -1& . & . & 1 & . \\
 21658^+      & . &  1 & . & . & . & . & . & . & . & . & . & . & 1 \\
 21658^-      & . &  1 & . & . & . & . & . & . & . & . & . & . & 1 \\
 22932^+      & 1 &  . & . & . & -1& 1 & . & . & . & 1 & -1& 2 & . \\
 22932^-      & 1 &  . & . & . & -1& . & 1 & . & 1 & . & -1& 2 & . \\
 46410^0\,    & 1 & -1 & . & . & -2& 1 & 1 & -1& 1 & 1 & -1& 2 & . \\
 63700^+      & . &  1 & . & . & 1 & . & . & 1 & . & . & . & 1 & 1 \\
 63700^-      & . &  1 & . & . & 1 & . & . & 1 & . & . & . & 1 & 1 \\
 198900^0\,   & . &  1 & . & . & 1 & . & . & 1 & . & . & 1 & . & 1 \\
 183600^+     & . &  . & 1 & . & . & 1 & . & 1 & . & . & . & 2 & . \\
 183600^-     & . &  . & . & 1 & . & . & 1 & 1 & . & . & . & 2 & . \\
 433160^0\,   & . &  . & . & . & . & . & . & 1 & . & . & 1 & 1 & . \\
 249900^+     & . &  . & . & . & 1 & . & . & . & 1 & . & . & -1& . \\
 249900^-     & . &  . & . & . & 1 & . & . & . & . & 1 & . & -1& . \\
 270725^+     & . &  . & . & . & . & . & 1 & . & 1 & . & . & 1 & 1 \\
 270725^-     & . &  . & . & . & . & 1 & . & . & . & 1 & . & 1 & 1 \\
 1082900^0\,  & . &  . & . & . & . & . & . & . & . & . & 1 & -1& . \\
 1082900^+    & . &  . & . & . & . & . & . & . & . & . & . & . & -1\\
 1082900^-    & . &  . & . & . & . & . & . & . & . & . & . & . & -1\\
 2598960^0\,  & . &  . & . & . & . & . & . & . & . & . & . & . & . \\
 3898440^0\,  & . &  . & . & . & . & . & . & . & . & . & . & 1 & . \\
 3411968^+    & . &  . & . & . & . & . & . & . & . & . & . & . & 1 \\
 3411968^-    & . &  . & . & . & . & . & . & . & . & . & . & . & 1 \\
 16777216^+   & . &  . & . & . & . & . & . & . & . & . & . & . & . \\
 16777216^-   & . &  . & . & . & . & . & . & . & . & . & . & . & . 
\\ \hline\hline
\end{array}
}
$$
\end{table}

\addtocounter{table}{-1}
\begin{table}
\caption{\label{Mod3B1RelsII} (cont.) The relations in the principal 
$3$-block with respect to the basic set $\text{\rm BS}_1$}
$$
{\footnotesize
\begin{array}{r|cccccccccccc} \hline\hline
%\chi & \begin{array}{c} 56 \\ 57 \end{array} % 5657600
%     & \begin{array}{c} 59 \\ 60 \end{array} % 7796880
%     & \begin{array}{c} 61 \\ 62 \end{array} % 8663200
%     & \begin{array}{c} 67 \\ 68 \end{array} % 9052160
%     & 72^+                                  % 5870592
%     & 72^-                                  % 5870592
%     & \begin{array}{c} 73 \\ 74 \end{array} % 12994800
%     & \begin{array}{c} 75 \\ 76 \end{array} % 14619150
%     & \begin{array}{c} 88 \\ 89 \end{array} % 23761920
%     & \begin{array}{c} 91 \\ 92 \end{array} % 29238300
%     & 95^+                                  % 17326400
%     & 95^-                                  % 17326400
\chi & \chi_{56}^0 & \chi_{59}^0 & \chi_{61}^0 & \chi_{67}^0 &
       \chi_{72}^+ & \chi_{72}^- & \chi_{73}^0 & \chi_{75}^0 &
       \chi_{88}^0 & \chi_{91}^0 & \chi_{95}^+ & \chi_{95}^- 
\rule[- 4pt]{0pt}{ 14pt} \\ \hline
 1^+          & . & 1 & . & 1 & . & 1 & 1 & 1 & . & . & 1 & 1 \rule[ 0pt]{0pt}{ 10pt} \\ 
 1^-          & . & 1 & . & 1 & 1 & . & 1 & 1 & . & . & 1 & 1 \\ 
 833^+        & -1& -2& . & -1& -1& -1& -2& -2& . & . & -1& -1 \\ 
 833^-        & -1& -2& . & -1& -1& -1& -2& -2& . & . & -1& -1 \\ 
 2210^0\,     & -1& -2& . & -1& -1& -1& -2& -2& . & . & -1& -1 \\ 
 1326^+       & 1 & 1 & -1& . & 1 & . & 1 & . & . & -1& -1& . \\ 
 1326^-       & 1 & 1 & -1& . & . & 1 & 1 & . & . & -1& . & -1 \\ 
 21658^+      & . & -1& 1 & . & -1& . & . & 1 & -1& -1& . & -1 \\ 
 21658^-      & . & -1& 1 & . & . & -1& . & 1 & -1& -1& -1& . \\ 
 22932^+      & 1 & 1 & . & 1 & 1 & 1 & 2 & 2 & . & -1& 1 & . \\ 
 22932^-      & 1 & 1 & . & 1 & 1 & 1 & 2 & 2 & . & -1& . & 1 \\ 
 46410^0\,    & 1 & 2 & . & 1 & 1 & 1 & 2 & 2 & . & -1& 1 & 1 \\ 
 63700^+      & 1 & . & . & . & . & . & 1 & 1 & -1& -2& -1& -1 \\ 
 63700^-      & 1 & . & . & . & . & . & 1 & 1 & -1& -2& -1& -1 \\ 
 198900^0\,   & 1 & . & . & . & . & . & . & . & -1& -1& -1& -1 \\ 
 183600^+     & 1 & 2 & -1& 1 & 2 & . & 2 & 1 & . & -1& -1& 1 \\ 
 183600^-     & 1 & 2 & -1& 1 & . & 2 & 2 & 1 & . & -1& 1 & -1 \\ 
 433160^0\,   & 1 & 2 & -1& 1 & 1 & 1 & 1 & . & . & . & . & . \\ 
 249900^+     & . & -1& . & -1& . & -1& -1& -1& . & . & -1& . \\ 
 249900^-     & . & -1& . & -1& -1& . & -1& -1& . & . & . & -1 \\ 
 270725^+     & 1 & 1 & . & . & . & 1 & 1 & . & . & -1& . & . \\ 
 270725^-     & 1 & 1 & . & . & 1 & . & 1 & . & . & -1& . & . \\ 
 1082900^0\,  & . & . & . & . & . & . & -1& -1& . & 1 & . & . \\ 
 1082900^+    & . & 1 & . & 1 & 1 & . & 1 & 1 & . & . & . & 1 \\ 
 1082900^-    & . & 1 & . & 1 & . & 1 & 1 & 1 & . & . & 1 & . \\ 
 2598960^0\,  & . & . & 1 & 1 & . & . & 1 & 1 & -1& . & . & . \\ 
 3898440^0\,  & 1 & 1 & . & 1 & 1 & 1 & 2 & 1 & . & -1& . & . \\ 
 3411968^+    & . & . & 1 & . & . & . & . & 1 & -1& . & . & . \\ 
 3411968^-    & . & . & 1 & . & . & . & . & 1 & -1& . & . & . \\ 
 16777216^+   & . & . & . & . & . & . & . & . & 1 & 1 & 1 & . \\ 
 16777216^-   & . & . & . & . & . & . & . & . & 1 & 1 & . & 1 
\\ \hline\hline
\end{array}
}
$$
\end{table}

\begin{table}
\caption{\label{Mod3B1ApproxI} The parametrized basic set 
decomposition matrix of the principal $3$-block}
$$
{
\begin{array}{r|rrrrrrrrrrrrrrrr} \hline\hline
\Phi & 1 & 2 & 3 & 4 & 5 & 6 & 7 & 8 & 9 & 10 & 11 & 12 & 13 & 14 & 15 \rule[- 3pt]{0pt}{ 15pt} \\ \hline
        1^+   & 1 & . & . & . & . & . & . & . & . & . & . & . & . & . & . \rule[ 0pt]{0pt}{ 10pt} \\
        1^-   & . & 1 & . & . & . & . & . & . & . & . & . & . & . & . & . \\
      833^+   & . & . & 1 & . & . & . & . & . & . & . & . & . & . & . & . \\
      833^-   & . & . & . & 1 & . & . & . & . & . & . & . & . & . & . & . \\
     2210^0\, & . & . & . & . & 1 & . & . & . & . & . & . & . & . & . & . \\
     1326^+   & . & . & . & . & . & 1 & . & . & . & . & . & . & . & . & . \\
     1326^-   & . & . & . & . & . & . & 1 & . & . & . & . & . & . & . & . \\
    21658^+   & . & . & . & . & . & . & . & 1 & . & . & . & . & . & . & . \\
    21658^-   & . & . & . & . & . & . & . & . & 1 & . & . & . & . & . & . \\
    22932^+   & . & . & . & . & 1 & . & . & . & . & 1 & . & . & . & . & . \\
    22932^-   & . & . & . & . & 1 & . & . & . & . & . & 1 & . & . & . & . \\
    46410^0\, & . & . & 1 & 1 & . & . & . & . & . & . & . & 1 & . & . & . \\
%   88400     & 1 & 1 & . & . & 1 & . & . & . & . & 1 & 1 & 1 & . & . & . \\
    63700^+   & . & . & . & . & . & . & . & . & . & . & . & . & 1 & . & . \\
    63700^-   & . & . & . & . & . & . & . & . & . & . & . & . & . & 1 & . \\
   198900^0\, & 1 & 1 & . & . & . & . & . & . & . & . & . & 1 & . & . & 1 \\
%  324870     & . & . & . & . & . & . & . & 1 & 1 & . & . & . & 1 & 1 & 1 \\
   183600^+   & . & . & . & . & . & . & . & . & . & . & . & . & . & . & . \\
   183600^-   & . & . & . & . & . & . & . & . & . & . & . & . & . & . & . \\
%  183600     & . & . & . & . & . & . & . & . & . & . & . & . & . & . & . \\
%  183600     & . & . & . & . & . & . & . & . & . & . & . & . & . & . & . \\
   433160^0\, & . & . & 1 & 1 & . & . & . & . & . & . & . & . & . & . & . \\
   249900^+   & 1 & . & 1 & . & . & . & 1 & . & . & 1 & . & 1 & . & . & . \\
   249900^-   & . & 1 & . & 1 & . & 1 & . & . & . & . & 1 & 1 & . & . & . \\
   270725^+   & . & . & . & . & . & . & . & . & . & . & . & . & . & . & . \\
   270725^-   & . & . & . & . & . & . & . & . & . & . & . & . & . & . & . \\
%  696150     & . & . & . & . & . & 2 & 2 & . & . & . & . & 1 & 1 & 1 & 1 \\
%  519792     & x & x & . & . & . & . & . & . & . & x & x & 1 & . & . & . \\
%  519792     & x & x & . & . & . & . & . & . & . & x & x & 1 & . & . & . \\
  1082900^0\, & . & . & . & . & 2 & . & . & . & . & 1 & 1 & 1 & . & . & . \\
% 1082900     & . & . & . & . & . & 1 & 1 & . & . & . & . & . & 1 & 1 & 1 \\
%  584766     & x & x & 1 & . & . & . & . & . & . & x & x & 2 & . & . & . \\
%  584766     & x & x & . & 1 & . & . & . & . & . & x & x & 2 & . & . & . \\
% 1624350     & . & . & 1 & 1 & . & . & . & . & . & . & . & 1 & . & . & 1 \\
  1082900^+   & . & . & . & . & . & . & . & 1 & . & . & . & . & . & . & . \\
  1082900^-   & . & . & . & . & . & . & . & . & 1 & . & . & . & . & . & . \\
  2598960^0\, & . & . & . & . & . & 1 & 1 & . & . & . & . & . & 1 & 1 & 1 \\
  3898440^0\, & . & . & . & . & . & . & . & . & . & . & . & 1 & . & . & . \\
% 4331600     & . & . & . & . & . & . & . & . & . & . & . & . & 1 & 1 & . \\
% 5569200     & . & . & 1 & 1 & . & . & . & . & . & . & . & 1 & 1 & 1 & 1 \\
% 5657600     & 1 & 1 & 1 & 1 & 1 & 1 & 1 & . & . & 1 & 1 & 3 & 1 & 1 & 1 \\
  3411968^+   & . & . & 1 & . & . & . & . & . & . & . & . & . & . & . & . \\
  3411968^-   & . & . & . & 1 & . & . & . & . & . & . & . & . & . & . & . \\
% 7796880     & . & . & 1 & 1 & . & . & . & . & . & . & . & 1 & . & . & . \\
% 8663200     & . & . & . & . & . & . & . & 1 & 1 & . & . & . & 1 & 1 & 1 \\
% 9052160     & . & . & . & . & 1 & . & . & 1 & 1 & . & . & . & 1 & 1 & 1 \\
% 5870592     & . & . & 1 & . & 1 & . & . & . & . & x & x & 1 & . & . & . \\
% 5870592     & . & . & . & 1 & 1 & . & . & . & . & x & x & 1 & . & . & . \\
%12994800     & . & . & . & . & . & 1 & 1 & 1 & 1 & . & . & 1 & 2 & 2 & 1 \\
%14619150     & . & . & . & . & . & . & . & 2 & 2 & . & . & . & 2 & 2 & 1 \\
%23761920     & . & . & . & . & . & 1 & 1 & . & . & . & . & 1 & 2 & 2 & . \\
%29238300     & . & . & . & . & . & 1 & 1 & . & . & . & . & . & 2 & 2 & 1 \\
 16777216^+   & 1 & . & 1 & . & . & 1 & 1 & 1 & . & . & . & 1 & 2 & 2 & 1 \\
 16777216^-   & . & 1 & . & 1 & . & 1 & 1 & . & 1 & . & . & 1 & 2 & 2 & 1 %\
%17326400     & . & . & . & . & . & . & . & 1 & . & . & . & . & 1 & 1 & . \\
%17326400     & . & . & . & . & . & . & . & . & 1 & . & . & . & 1 & 1 & . 
\\ \hline\hline
\end{array}
}
$$
\end{table}

\addtocounter{table}{-1}
\begin{table}
\caption{\label{Mod3B1ApproxII} (cont.) The parametrized basic set
decomposition matrix of the principal $3$-block}
$$
{
\begin{array}{r|rrrrrrrrrrrrrrrrr} \hline\hline
\Phi & 16 & 17 & 18 & 19 & 20 & 21 & 22 & 23 & 24 & 25 & 26 & 27 & 28 & 29 & 30 & 31 \rule[- 3pt]{0pt}{ 15pt} \\ \hline
          1^+   & . & . & . & . & . & . & . & . & . & . & . & . & . & . & . & . \rule[ 0pt]{0pt}{ 10pt} \\
          1^-   & . & . & . & . & . & . & . & . & . & . & . & . & . & . & . & . \\
        833^+   & . & . & . & . & . & . & . & . & . & . & . & . & . & . & . & . \\
        833^-   & . & . & . & . & . & . & . & . & . & . & . & . & . & . & . & . \\
       2210^0\, & . & . & . & . & . & . & . & . & . & . & . & . & . & . & . & . \\
       1326^+   & . & . & . & . & . & . & . & . & . & . & . & . & . & . & . & . \\
       1326^-   & . & . & . & . & . & . & . & . & . & . & . & . & . & . & . & . \\
      21658^+   & . & . & . & . & . & . & . & . & . & . & . & . & . & . & . & . \\
      21658^-   & . & . & . & . & . & . & . & . & . & . & . & . & . & . & . & . \\
      22932^+   & . & . & . & . & . & . & . & . & . & . & . & . & . & . & . & .   \\
      22932^-   & . & . & . & . & . & . & . & . & . & . & . & . & . & . & . & .   \\
      46410^0\, & . & . & . & . & . & . & . & . & . & . & . & . & . & . & . & .   \\
%     88400     & . & . & . & . & . & . & . & . & . & . & . & . & . & . & . & .   \\
      63700^+   & . & . & . & . & . & . & . & . & . & . & . & . & . & . & . & .   \\
      63700^-   & . & . & . & . & . & . & . & . & . & . & . & . & . & . & . & .   \\
     198900^0\, & . & . & . & . & . & . & . & . & . & . & . & . & . & . & . & .   \\
%    324870     & . & . & . & . & . & . & . & . & . & . & . & . & . & . & . & .   \\
     183600^+   & 1 & . & . & . & . & . & . & . & . & . & . & . & . & . & . & .   \\
     183600^-   & . & 1 & . & . & . & . & . & . & . & . & . & . & . & . & . & .   \\
%    183600     & 1 & . & . & . & . & . & . & . & . & . & . & . & . & . & . & .   \\
%    183600     & . & 1 & . & . & . & . & . & . & . & . & . & . & . & . & . & .   \\
     433160^0\, & . & . & 1 & . & . & . & . & . & . & . & . & . & . & . & . & .   \\
     249900^+   & . & . & . & 1 & . & . & . & . & . & . & . & . & . & . & . & .   \\
     249900^-   & . & . & . & . & 1 & . & . & . & . & . & . & . & . & . & . & .   \\
     270725^+   & . & . & . & . & . & 1 & . & . & . & . & . & . & . & . & . & .   \\
     270725^-   & . & . & . & . & . & . & 1 & . & . & . & . & . & . & . & . & .   \\
%    696150     & . & . & . & 1 & 1 & . & . & . & . & . & . & . & . & . & . & .   \\
%    519792     & 1 & . & . & . & . & x & x & . & . & . & . & . & . & . & . & .   \\
%    519792     & . & 1 & . & . & . & x & x & . & . & . & . & . & . & . & . & .   \\
    1082900^0\, & . & . & . & . & . & . & . & 1 & . & . & . & . & . & . & . & .   \\
%   1082900     & 1 & 1 & 1 & . & . & . & . & . & . & . & . & . & . & . & . & .   \\
%    584766     & . & . & . & 1 & . & x & x & . & . & . & . & . & . & . & . & .   \\
%    584766     & . & . & . & . & 1 & x & x & . & . & . & . & . & . & . & . & .   \\
%   1624350     & . & . & 1 & . & . & . & . & 1 & . & . & . & . & . & . & . & .   \\
    1082900^+   & . & . & . & . & . & . & . & . & 1 & . & . & . & . & . & . & .   \\
    1082900^-   & . & . & . & . & . & . & . & . & . & 1 & . & . & . & . & . & .   \\
    2598960^0\, & . & . & . & . & . & . & . & . & . & . & 1 & . & . & . & . & .   \\
    3898440^0\, & . & . & . & 1 & 1 & . & . & 1 & . & . & . & 1 & . & . & . & .   \\
%   4331600     & 2 & 2 & 1 & . & . & 1 & 1 & . & . & . & . & 1 & . & . & . & .   \\
%   5569200     & 2 & 2 & 2 & . & . & 1 & 1 & . & . & . & . & . & 1 & 1 & . & .   \\
%   5657600     & 1 & 1 & 1 & 1 & 1 & 1 & 1 & 1 & . & . & . & 1 & . & . & . & .   \\
    3411968^+   & 2 & . & 1 & . & . & . & . & . & 1 & . & . & . & 1 & . & . & .   \\
    3411968^-   & . & 2 & 1 & . & . & . & . & . & . & 1 & . & . & . & 1 & . & .   \\
%   7796880     & 2 & 2 & 2 & . & . & 1 & 1 & 1 & 1 & 1 & . & 1 & . & . & . & .   \\
%   8663200     & 1 & 1 & 1 & . & . & . & . & . & 1 & 1 & 1 & . & 1 & 1 & . & .   \\
%   9052160     & 1 & 1 & 1 & . & . & . & . & 1 & 1 & 1 & 1 & 1 & . & . & . & .   \\
%   5870592     & 2 & . & 1 & 1 & . & x & x & 1 & 1 & . & . & 1 & . & . & . & .   \\
%   5870592     & . & 2 & 1 & . & 1 & x & x & 1 & . & 1 & . & 1 & . & . & . & .   \\
%  12994800     & 2 & 2 & 1 & 1 & 1 & 1 & 1 & 1 & 1 & 1 & 1 & 2 & . & . & . & .   \\
%  14619150     & 3 & 3 & 2 & . & . & . & . & . & 2 & 2 & 1 & 1 & 1 & 1 & . & .   \\
%  23761920     & 2 & 2 & . & 1 & 1 & 1 & 1 & . & 1 & 1 & 1 & 2 & . & . & 1 & 1  \\
%  29238300     & 3 & 3 & 2 & . & . & . & . & . & 2 & 2 & 2 & 1 & 1 & 1 & 1 & 1  \\
   16777216^+   & a & \tilde{a} & 1 & 1 & . & 1 & . & . & 2 & . & 1 & 1 & 1 & . & 1 & .   \\
   16777216^-   & \tilde{a} & a & 1 & . & 1 & . & 1 & . & . & 2 & 1 & 1 & . & 1 & . & 1   %\\
%  17326400     & x & x & 1 & . & . & x & x & . & 2 & 1 & 1 & 1 & 1 & . & 1 & .   \\
%  17326400     & x & x & 1 & . & . & x & x & . & 1 & 2 & 1 & 1 & . & 1 & . & 1 
\\ \hline\hline
\end{array}
}
$$
\end{table}

\begin{table}
\caption{\label{Mod3B1Projs} Some projective characters of the 
principal $3$-block}
$$
{
\begin{array}{r|rrrrrrrrr} \hline\hline
\Phi & 32 & 33 & 34 & 35 & 36 & 37 & 38 & 39 & 40 \rule[- 3pt]{0pt}{ 15pt} \\ \hline
        1^+   & 1 & . & . & . & . &  . & . & . & . \rule[ 0pt]{0pt}{ 10pt} \\ 
        1^-   & . & . & . & . & . &  . & . & . & . \\ 
      833^+   & . & 1 & . & . & . &  . & . & . & . \\ 
      833^-   & . & . & . & . & . &  . & . & . & . \\ 
     2210^0\, & 2 & . & . & . & . &  . & . & . & . \\ 
     1326^+   & . & . & 1 & . & . &  . & . & . & . \\ 
     1326^-   & . & . & . & . & . &  . & . & . & . \\ 
    21658^+   & . & . & . & 1 & . &  . & . & . & . \\ 
    21658^-   & . & . & . & . & . &  . & . & . & . \\ 
    22932^+   & 5 & . & . & . & 1 &  . & . & . & . \\ 
    22932^-   & 4 & . & . & . & . &  . & . & . & . \\ 
    46410^0\, & . & 1 & . & . & . &  . & . & . & . \\ 
    63700^+   & . & . & . & . & . &  . & . & . & . \\ 
    63700^-   & . & . & . & . & . &  . & . & . & . \\ 
   198900^0\, & 1 & . & . & . & . &  . & . & . & . \\ 
   183600^+   & . & . & . & . & . &  1 & . & . & . \\ 
   183600^-   & . & . & . & . & . &  . & . & . & . \\ 
   433160^0\, & . & 1 & . & . & . &  . & . & . & . \\ 
   249900^+   & 4 & 1 & . & . & 1 &  . & 1 & . & . \\ 
   249900^-   & 2 & . & 1 & . & . &  . & . & . & . \\ 
   270725^+   & . & . & . & . & 1 &  . & . & . & . \\ 
   270725^-   & . & . & . & . & . &  . & . & . & . \\ 
  1082900^0\, & 9 & . & . & . & 2 &  . & . & . & . \\ 
  1082900^+   & . & . & . & 1 & . &  . & . & 1 & . \\ 
  1082900^-   & . & . & . & . & . &  . & . & . & . \\ 
  2598960^0\, & . & . & 1 & . & . &  1 & . & . & . \\ 
  3898440^0\, & . & . & . & . & 1 &  1 & 1 & . & . \\ 
  3411968^+   & . & 2 & . & . & . &  3 & . & 1 & 1 \\ 
  3411968^-   & . & . & . & 1 & . &  . & . & . & . \\ 
 16777216^+   & 1 & 2 & 1 & 1 & 1 & 10 & 2 & 2 & 1 \\ 
 16777216^-   & . & . & 1 & 1 & . &  6 & . & . & . 
\\ \hline\hline
\end{array}
}
$$
\end{table}

%\addtocounter{table}{-1}
%\renewcommand{\baselinestretch}{0.9}
\begin{table}
\caption{\label{Mod3B2Decmat} The decomposition matrix of
the $3$-block $B_2$}
$
{
\begin{array}{r|rrrrrrr} \hline\hline
\Phi & 1 & 2 & 3 & 4 & 5 & 6 & 7 \rule[- 3pt]{0pt}{ 15pt} \\ \hline
      1377^+   & 1 & . & . & . & . & . & . \rule[ 0pt]{0pt}{ 10pt} \\
      1377^-   & . & 1 & . & . & . & . & . \\       
    269892^+   & . & . & 1 & . & . & . & . \\       
    269892^-   & . & . & . & 1 & . & . & . \\       
    716040^0\, & 1 & 1 & . & . & 1 & . & . \\       
   1253070^0\, & . & . & 1 & 1 & 1 & . & . \\       
  10024560^0\, & . & . & . & . & 1 & 1 & 1 \\       
   5640192^+   & 1 & . & 1 & . & 1 & 1 & . \\       
   5640192^-   & . & 1 & . & 1 & 1 & . & 1        
\\ \hline\hline
\end{array}
}
$
\end{table}

%$\quad\quad
%${
\bigskip
\bigskip

\begin{table}
\caption{\label{Mod3B9Decmat} The parametrized decomposition matrix of
the $3$-block $B_9$}
$
{
\begin{array}{r|rrrrrrr} \hline\hline
\Phi & 1 & 2 & 3 & 4 & 5 & 6 & 7 \rule[- 3pt]{0pt}{ 15pt} \\ \hline
    877149^+   & 1 & . & . & . & . & . & . \rule[ 0pt]{0pt}{ 10pt} \\
    877149^-   & . & 1 & . & . & . & . & . \\
    877149^+   & . & . & 1 & . & . & . & . \\
    877149^-   & . & . & . & 1 & . & . & . \\
   8771490^0\, & 1 & 1 & . & . & 1 & . & . \\
   8771490^0\, & . & . & 1 & 1 & 1 & . & . \\
  17542980^0\, & . & . & . & . & 1 & 1 & 1 \\
  14034384^+   & b & \tilde{b} & 1 & . & 1 & 1 & . \\
  14034384^-   & \tilde{b} & b & . & 1 & 1 & . & 1
\\ \hline\hline
\end{array}
}
$
\end{table}

\begin{table}
\caption{\label{Mod7B1ApproxI} The parametrized decomposition matrix of the principal $7$-block}
$$
{
\begin{array}{r|rrrrrrrrrrrrrrrr} \hline\hline
\Phi & 1 & 2 & 3 & 4 & 5 & 6 & 7 & 8 & 9 & 10 & 11 & 12 & 13 & 14 & 15 & 16 \rule[- 3pt]{0pt}{ 15pt} \\ \hline
       1^+   & 1 & . & . & . & . & . & . & . & . & . & . & . & . & . & . & . \rule[ 0pt]{0pt}{ 10pt} \\
       1^-   & . & 1 & . & . & . & . & . & . & . & . & . & . & . & . & . & . \\
    2210^0\, & 1 & 1 & 1 & . & . & . & . & . & . & . & . & . & . & . & . & . \\
    1326^+   & . & . & . & 1 & . & . & . & . & . & . & . & . & . & . & . & . \\
    1326^-   & . & . & . & . & 1 & . & . & . & . & . & . & . & . & . & . & . \\
    1377^+   & . & . & . & . & . & 1 & . & . & . & . & . & . & . & . & . & . \\
    1377^-   & . & . & . & . & . & . & 1 & . & . & . & . & . & . & . & . & . \\
   88400^0\, & . & . & . & . & . & 1 & 1 & 1 & . & . & . & . & . & . & . & . \\
  198900^0\, & . & . & 1 & . & . & . & . & . & 1 & . & . & . & . & . & . & . \\
  183600^+   & . & . & . & . & . & . & . & . & . & 1 & . & . & . & . & . & . \\
  183600^-   & . & . & . & . & . & . & . & . & . & . & 1 & . & . & . & . & . \\
  183600^+   & . & . & . & . & . & . & . & . & . & . & . & 1 & . & . & . & . \\
  183600^-   & . & . & . & . & . & . & . & . & . & . & . & . & 1 & . & . & . \\
  322218^+   & 1 & . & 1 & . & . & . & . & . & . & . & . & . & . & 1 & . & . \\
  322218^-   & . & 1 & 1 & . & . & . & . & . & . & . & . & . & . & . & 1 & . \\
  716040^0\, & . & . & . & . & . & . & . & 1 & . & . & . & . & . & . & . & 1 \\
  947700^+   & . & . & . & . & . & a & \tilde{a} & 1 & . & . & . & . & . & . & . & . \\
  947700^-   & . & . & . & . & . & \tilde{a} & a & 1 & . & . & . & . & . & . & . & . \\
 2685150^0\, & . & . & . & 1 & 1 & . & . & . & 1 & . & . & . & . & . & . & . \\
 5657600^0\, & . & . & . & . & . & . & . & 1 & . & 1 & 1 & 1 & 1 & . & . & 1 \\
 9052160^0\, & . & . & 1 & . & . & . & . & . & 1 & . & . & . & . & 1 & 1 & . \\
 5640192^+   & . & . & . & 1 & . & . & . & . & . & b & \tilde{b} & b & \tilde{b} & . & . & . \\
 5640192^-   & . & . & . & . & 1 & . & . & . & . & \tilde{b} & b & \tilde{b} & b & . & . & . \\
16110900^0\, & . & . & . & . & . & . & . & . & 1 & 1 & 1 & 1 & 1 & . & . & 1 \\
21481200^0\, & . & . & . & . & . & . & . & . & . & 1 & 1 & 1 & 1 & . & . & 1 \\
16777216^+   & . & . & . & . & . & . & . & . & . & c & \tilde{c} & c & \tilde{c} & 1 & . & . \\
16777216^-   & . & . & . & . & . & . & . & . & . & \tilde{c} & c & \tilde{c} & c & . & 1 & . 
\\ \hline\hline
\end{array}
}
$$
\end{table}

\addtocounter{table}{-1}
\begin{table}
\caption{\label{Mod7B1ApproxII} (cont.) The parametrized decomposition matrix of the principal $7$-block}
$$
{
\begin{array}{r|rrrrrrrr|rr} \hline\hline
\Phi & 17 & 18 & 19 & 20 & 21 & 22 & 23 & 24 & 25 & 26
\rule[- 3pt]{0pt}{ 15pt} \\ \hline
       1^+   & . & . & . & . & . & . & . & . & . & . \rule[ 0pt]{0pt}{ 10pt} \\
       1^-   & . & . & . & . & . & . & . & . & . & . \\
    2210^0\, & . & . & . & . & . & . & . & . & . & . \\
    1326^+   & . & . & . & . & . & . & . & . & . & . \\
    1326^-   & . & . & . & . & . & . & . & . & . & . \\
    1377^+   & . & . & . & . & . & . & . & . & 1 & . \\
    1377^-   & . & . & . & . & . & . & . & . & . & . \\
   88400^0\, & . & . & . & . & . & . & . & . & 1 & . \\
  198900^0\, & . & . & . & . & . & . & . & . & . & . \\
  183600^+   & . & . & . & . & . & . & . & . & . & 1 \\
  183600^-   & . & . & . & . & . & . & . & . & . & . \\
  183600^+   & . & . & . & . & . & . & . & . & . & 1 \\
  183600^-   & . & . & . & . & . & . & . & . & . & . \\
  322218^+   & . & . & . & . & . & . & . & . & . & . \\
  322218^-   & . & . & . & . & . & . & . & . & . & . \\
  716040^0\, & . & . & . & . & . & . & . & . & . & . \\
  947700^+   & 1 & . & . & . & . & . & . & . & 1 & . \\
  947700^-   & . & 1 & . & . & . & . & . & . & 1 & . \\
 2685150^0\, & . & . & 1 & . & . & . & . & . & . & 1 \\
 5657600^0\, & 1 & 1 & 1 & . & . & . & . & . & 1 & 3 \\
 9052160^0\, & . & . & . & 1 & . & . & . & . & . & 4 \\
 5640192^+   & d & \tilde{d} & 1 & . & 1 & . & . & . & 1 & 4 \\
 5640192^-   & \tilde{d} & d & 1 & . & . & 1 & . & . & . & 2 \\
16110900^0\, & . & . & 1 & 1 & 1 & 1 & . & . & . & 9 \\
21481200^0\, & . & . & . & 1 & . & . & 1 & 1 & . & 14\\
16777216^+   & . & . & . & 1 & 1 & . & 1 & . & . & 11\\
16777216^-   & . & . & . & 1 & . & 1 & . & 1 & . & 9 
\\ \hline\hline
\end{array}
}
$$
\end{table}